\documentclass[10pt]{article}
\usepackage{amsmath,amssymb,amsthm,amscd}
\numberwithin{equation}{section}

\def\cM{{\mathcal M}}
\def\cN{{\mathcal N}}

\def\bG{{\mathbf G}}

\def\bE{{\mathbf E}}

\def\bF{{\mathbf F}}

\def\bK{{\mathbf K}}
\def\bT{{\mathbf T}}
\def\bF{{\mathbf F}}
\def\bJ{{\mathbf J}}

\def\bU{{\mathbf U}}
\def\bV{{\mathbf V}}
\def\bW{{\mathbf W}}

\def\bU{{\mathbf U}}
\def\bN{{\mathbf N}}
\def\bM{{\mathbf M}}

\def\be{{\mathbf e}}

\def\bw{{\mathbf w}}
\def\GL{{\mathbf G}{\mathbf L}}
\def\SL{{\mathbf S}{\mathbf L}}
\def\gl{{\mathbf g}{\mathbf l}}

\def\ZZ{{\mathbb Z}}

\def\RR{{\mathbb R}}
\def\CC{{\mathbb C}}

\newtheorem{prop}{Proposition}[section]
\newtheorem{theo}[prop]{Theorem}
\newtheorem{lemm}[prop]{Lemma}
\newtheorem{coro}[prop]{Corollary}
\newtheorem{rema}[prop]{Remark}
\newtheorem{exam}[prop]{Example}
\newtheorem{defi}[prop]{Definition}
\newtheorem{conj}[prop]{Conjecture}
\def\begeq{\begin{equation}}
\def\endeq{\end{equation}}

\def\bN{{\mathbf N}}

\title{Stability of pairs}
\author{Gang Tian\thanks{Supported partially by
a NSF grant}\\Beijing University and Princeton University}
\date{}

\begin{document}

\maketitle

\noindent
{\bf Abstract}: This is an expository note based on S. Paul's works on the stability of pairs (\cite{paul12a}, \cite{paul12b}, \cite{paul13}, \cite{paul08}).

\tableofcontents
\section{Introduction}

In this note, we discuss the stability of pairs and its related topics, mostly due to S. Paul.
\cite{paul12a}, \cite{paul12b} and \cite{paul13}, motivated by his study of the K-stability, he introduced the notion of the stability of pairs
and reformulated the K-stability, which I introduced in the middle of 90s. This formulation fits better with the Geometric Invariant Theory
and enables us to extend the arguments from the Geometric Invariant Theory to proving theorems on the K-stability we expected, such as,
an extension of the Hilbert-Mumford criterion. As a consequence, we provide detailed arguments for an approach suggested in \cite{tian10}
and used in \cite{tian12} as an alternative proof for the existence of K\"ahler-Einstein metrics on K-stable Fano manifolds assuming
the partial $C^0$-estimate. In the end, we will propose a question on the moduli of semistable pairs.

I would like to thank S. Paul for those discussions
on stability of pairs in the summer of 2012. During those discussions, he showed me his ideas and explained his reasoning. They are
very helpful and made it clear to me most results of this note. 

\section{Stability of pairs}
In this section, we recall basic definitions.
Let $\bG$ be one of the classical subgroups of $\GL(N+1,\CC)$, for example, take $\bG\,=\,\SL(N+1,\CC)$.
Let $\bV$ be a rational representation of $\bG$ \footnote{All representations in this note are finite dimensional and complex.}. The rationality means
that for all $\alpha\in \bV^{\vee}$ (dual space) and $v\in \bV\setminus \{0\}$ the {matrix coefficient} $\varphi_{\alpha , v}$
is a {regular function} on $\bG$, that is, $\varphi_{\alpha , v}\in \CC[\bG]$, where
\begin{equation}
\varphi_{\alpha , v}: \,\bG\,\mapsto  \,\CC,~~~~ \varphi_{\alpha , v}(\sigma)\,=\,\alpha(\rho(\sigma) v).
\end{equation}

For any $v\in \bV\setminus\{0\}$,  we let $[v]$ be the line in $\mathbb{P}(\bV)$ corresponding to $v$ and $\bG [v]$
be orbit of $[v]$ in $\mathbb{P}(\bV)$. Given a pair  $v\in\bV\setminus\{0\}$ and $w\in \bW\setminus\{0\}$, we have orbits
$\bG [v,w] \,\subset \,\mathbb{P}(\bV\oplus\bW)$ and $\bG [v,0]\,\subset\, \mathbb{P}(\bV\oplus\{0\})\,\subset \,\mathbb{P}(\bV\oplus\bW)$.
We will denote their closures by $\overline{\bG [v,w]}$ and $\overline{\bG [v,0]}$.

Following \cite{paul12a} and \cite{paul12b} (also see \cite{paul13}, we have the following.\footnote{Our definition for stable pairs differs from that in
\cite{paul13}. See Section 4 how they are related.}

\begin{defi}\label{defi:pair}
We say the pair $(v,w)$ is \emph{semistable} if
$$\overline{\bG [v,w]}\,\cap\,\overline{\bG [v,0]}\,=\,\emptyset.$$
We say $(v,w)$ is \emph{stable} if the stabilizer of $[v,w]$ in $\bG$ is finite and
$$\overline{\bG [v,w]}\,\backslash \,\bG [v,w]\,\subset\, \mathbb{P}(\{0\}\oplus\bW)\,\subset\, \mathbb{P}(\bV\oplus\bW).$$
\end{defi}

\begin{exam}\label{exam:hmss}
Let $\bV=\CC, v=1$ be the trivial 1-dimensional representation and $\bW$ be any rational representation of $\bG$.
Then $(1,w)$ is semistable if and only if $0$ is not in the closure of the affine orbit $\bG w$. Furthermore,
$(1,w)$ is stable if and only if the stabilizer of $w$ is finite and the orbit $\bG w$ is closed in $\bW$.
In other words, $w$ is semistable or stable in the usual sense of Geometric Invariant Theory.
\end{exam}

This example shows that Definition \ref{defi:pair} generalizes the notion of stability in classical Geometric Invariant Theory. Next we also extend the numerical
criterion, i.e., the Hilbert-Mumford criterion, to the case of pairs. First we fix some notations.

Let $\bT$ be a maximal algebraic torus of $\bG$. Let $\bM_{\ZZ}\,=\,\bM_{\ZZ}(\bT)$ denote the {character lattice} of $\bT$ defined by
\begin{equation}\label{eq:char-lattice}
\bM_{\ZZ}\,= \,\mbox{Hom}_{\ZZ}(\bT,\CC^*) .
\end{equation}
Its dual lattice is denoted by $\bN_{\ZZ}$. Each $ u\in \bN_{\ZZ}$ corresponds to an algebraic one parameter subgroup
$\lambda^u$ of $\bT$. More explicitly, the correspondence is given by
$$ m(\lambda^u(t))\,=\,t^{(u , m)},~~~~\forall t\in \bT,$$
where $ (\cdot \ , \ \cdot)$ is the standard pairing: $\bN_{\ZZ}\times\bM_{\ZZ}\,\mapsto\, \ZZ $.
As usual, we denote by associated real vector spaces
$$\bM_{\RR}\,=\, \bM_{\ZZ}\otimes_{\ZZ}\mathbb{R},~~~\bN_{\mathbb{R}}\,=\,\bN_{\ZZ}\otimes_{\ZZ}\mathbb{R}.$$

Since $\bV$ is rational, it decomposes under the action of $\bT$ into \emph{weight spaces}
\begin{equation}
\label{eq:decomp}
\bV\,=\,\bigoplus_{a\in {A}}\bV_{a},~~~{\rm where}~~{\bV}_{a}\,=\,\{v\in \bV \, |\, t\cdot v\,=\,a(t)\, v \ , \, t \in \bT\}.
\end{equation}
Here $A$ denotes the {support} of $\bV$, i.e.,
$$A\,=\, \{a \in \bM_{\ZZ}\ | \ \bV_{a}\neq 0\}.$$
Given $v\in \bV \setminus \{0\}$ , we denote by $v_a$ the projection of $v$ into $\bV_{a}$ and by $A(v)$ its support:
\begin{equation}\label{eq:support-v}
A(v)\,=\, \{\,a\,\in\, A\, | \, v_{a}\neq 0\,\}.
\end{equation}

\begin{defi}\label{defi:weight}
Let $\bT$ be any maximal torus in $\bG$ and $v\in \bV\setminus\{0\}$. We define the {\bf weight polytope} $\mathcal{N}(v)$ of $v$ to be the convex hull of $A(v)$ in
$\bM_{\RR}$. Furthermore, we define the {\bf weight} $w_{\lambda}(v)$ of $\lambda \in \bN_\ZZ$ to be the integer
\begin{equation}
\label{eq:weight}
w_{\lambda}(v)\,= \,\min_{x\in \cN(v) }  u (x)\,=\, \min_{ a \in A(v)}  (a,u ), ~~{\rm where}~\lambda \sim u\,\in \,\bN_{\ZZ} \ .
\end{equation}
\end{defi}
Observe that the weight of $\lambda$ can be characterized as the unique integer $w_{\lambda}(v)$ such that there is a non-zero limit $v_0$ in $\bV$:
\begin{equation}\label{eq:weight-2}
\lim_{\alpha\to 0}\alpha^{-w_{\lambda}(v)}\lambda(\alpha) v\,=\,v_0 \,\not=\,0.
\end{equation}

\begin{defi}\label{numerical}
Let $\bV$ and $\bW$ be two $\bG$-representations and $v\in\bV\setminus\{0\}, \ w\in \bW\setminus\{0\}$. We say $(v,w)$
\emph{numerically semistable} if $w_{\lambda}(w)\,\leq\, w_{\lambda}(v)$ for all one parameter subgroups $\lambda$ of $\bG$.
We say $(v,w)$
\emph{numerically stable} if $w_{\lambda}(w)\,<\, w_{\lambda}(v)$ for all one parameter subgroups $\lambda$ of $\bG$.
\end{defi}

\begin{exam}\label{exam:sl2}
Let $\bV_e$ and $\bV_d$ be irreducible $SL(2,\mathbb{C})$-representations with highest weights $e , d \in\mathbb{N}$.
These are isomorphic to spaces of homogeneous polynomials in two variables of degree $e$ and $d$. Let $f$ and $g$ be two such polynomials in $\bV_e\setminus\{0\}$ and $\bV_d\setminus\{0\}$
respectively. Then the pair $(f,g)$ is numerically semistable if and only if
\begin{equation}\label{d-e/2}
e\,\le\, d ~~{\rm and}~~\forall p\in \mathbb{P}^1,~ \mbox{ord}_p(g)\,-\,\mbox{ord}_p(f)\,\le\, \frac{d-e}{2}.
\end{equation}
\end{exam}

\section{Hilbert-Mumford-Paul criterion}
In this section, we prove a theorem which is essentially due to S. Paul\footnote{Actually, Paul proved only the semi-stable part of this theorem and an analogous result in Section 4 for the K-stability. However, his arguments work for the stable case, too.}
and extends the Hilbert-Mumford criterion in the Geometric Invariant Theory to pairs.

\begin{theo}\label{th:HMP}
Let $\bV$ and $\bW$ be two $\bG$-representations and $v\in\bV\setminus\{0\}, \ w\in \bW\setminus\{0\}$. Then
$(v,w)$ is stable (resp. semistable) if and only if it is numerically stable (resp. numerically semistable).
\end{theo}

\begin{proof}
We will prove only the stable case. The semistable case can be done in an identical and slightly simpler way.

First we prove the easy direction: Assume that $(v,w)$ is stable, we want to prove it is numerically stable.
Let $\lambda$ be any one parameter subgroup of $\bG$. Then by \eqref{eq:weight-2}, we have
$$\lim_{\alpha\to 0} \,\lambda(\alpha) [v,w]\,=\,\lim_{\alpha \to 0} \,\alpha^{-w_{\lambda}(v)} \lambda(\alpha) [v,w] \,=\,[v_0, \lim_{\alpha\to 0}
\alpha^{w_{\lambda}(w)-w_{\lambda}(v)} w_0].$$
By the stability, this limit should lie in $\mathbb{P}(\{0\}\oplus\bW)$, so $w_{\lambda}(w)\,< \,w_{\lambda}(v)$, consequently,
$(v,w)$ is numerically stable.

Now we will prove the converse by contradiction. Assume that $(v,w)$ is numerically stable but not stable.
It follows from the definition that there is $p$ outside $\mathbb{P}(\{0\}\oplus\bW)$ and in $\overline{\bG [v,w]}\,\backslash\, \bG [v,w]$.

Let $\bT$ be a maximal algebraic torus of $\bG$, then we have the Cartan decomposition:
$\bG\,=\,\bK \bT \bK$, where $\bK$ is a maximal compact subgroup of $\bG$.

The orbit $\bT p$ is contained in $\bE$ and its closure $\overline{\bT p} $ in $\bE$
contains a closed orbit, where
$$\bE\,=\, \mathbb{P}(\bV \oplus\bW)\backslash \mathbb{P}(\{0\}\oplus\bW).$$

Let $k_i, k_i'\in \bK$ and $t_i\in \bT$ with $\lim k_i t_i k_i' [v,w]\,=\,p$, by taking a subsequence if necessary, we may assume
that $k_i$ converge to $k\in \bK$, then
$$\lim_{i\to\infty} t_i k_i' [v,w] \,=\, k p \in \bE.$$
Therefore, by replacing $p$ by $kp$, we may assume
that $p$ is in the closure of $\bT \bK [v,w]$. Without loss of generality, we may further assume that $\bT p$ is closed in $\bE$.

\vskip 0.1in
\noindent
{\bf Claim}: {\it There is a $k\in \bK$ such that $\overline{\bT k [v,w]}\,\cap\, \bT p$ is non-empty}.
\begin{proof}
Assume this claim is false, that is, for any $k\in \bK$, $\overline{\bK \bT k [v,w]}\cap \bT p$ is non-empty.
Using \eqref{eq:decomp}, we get a $\bT$-invariant decomposition:
$$
\bV\,=\,\bigoplus_{a\in {A}}\bV_{a},~~~{\rm where}~~{\bV}_{a}\,=\,\{v\in \bV \, |\, t\cdot v\,=\,a(t)\, v \ , \, t \in \bT\}.
$$
Choose a basis $\{\be_\ell\}$ of $\bV$ such that each $\be_\ell$ lies in one of $\bV_a$'s. Then we define
\begin{equation}
\label{eq:hyperplane}
\bU_l\,=\, \mathbb{P}(\bV \oplus\bW)\,\backslash \,\mathbb{P}(\{ \, \sum_{i=1}^d x_i \be_i \,|\,x_i\in \RR,~x_l = 0 \,\}\oplus \bW),
\end{equation}
where $l = 1,\cdots, d=\dim \bV$. Each $\bU_l$
is a $\bT$-invariant affine open subvariety of $\mathbb{P}(\bV\oplus\bW)$ and is the complement of a hyperplane
containing $\mathbb{P}(\{0\}\oplus\bW)$. Clearly, $\bU_1, \cdots , \bU_d$ cover $\bE$.
For each $l$, $\bT p$ is either contained in or
disjoint from $ \bU_l$. Similarly, for each $k \in \bK$, $\bT k [v,w]$ is either contained in or disjoint from $\bU_l$.
For each $k \in \bK$, choose an $i(k)$ such that $\bT k [v,w]\subset \bU_{i(k)}$.
Since $\overline{\bT k [v,w]}\cap \bT p \cap \bU_{i(k)}\,=\,\emptyset$, there is a $\bT$-invariant polynomial $f_k$ on $\bU_{i(k)}$ satisfying:
$$f_k \,\equiv\,1~~{\rm on}~\overline{\bT k [v,w]}\cap U_{i(k)}~~{\rm and}~~f_k\,\equiv\,0~~{\rm on}~\bT p \cap \bU_{i(k)}.$$
Note that $\bT p \cap \bU_{i(k)}$ may be empty. Then there is a $r=r(k) > 0$ such that
$$\overline{B_r(k)}[v,w]\,\subset\, U_{i(k)}~~~{\rm and}~~~\overline{B_r(k)}[v,w]\cap f^{-1}_k(0) \,=\,\emptyset.$$
Since $\bK$ is compact, we can find $k_1,\cdots, k_a$ such that $B_{r_1}(k_1),\cdots, B_{r_a}(k_a)$, where $r_j = r(k_j)$, cover $\bK$.
Let $\{\eta_j\}$ be a partition of unit associated to the covering $\{B_{r_j}(k_j)\}$, then we define a function on $\bK$:
$$F(k) \,=\,\sum_{j=1}^a \,\eta_j(k)\,|f_{k_j}(k [v,w] )|.$$
Then $F\,\ge\, c$ for a positive constant $c > 0$. Since $p\in \overline{\bT\bK [v,w]}$, there is a sequence $ \{(t_i, k_i') \}$ in
$\bT\times \bK$
such that $t_ik_i'[v,w]$ converge to $p$, it follows
$$F(k_i')\,=\,\sum_{j=1}^a\, \eta(k_i') \,|f_{k_j}(t_i k_i'[v,w])| \,\rightarrow\, 0.$$
This is a contradiction, so {\bf Claim} is proved.

\end{proof}
The above claim gives a $\bU_l$ which contains both $\bT p$ and $\bT k [v,w]$. But $\bU_l$ can be identified with
the hyperplane of $\bV\oplus\bW$:
$$\{ \, \sum_{i=1, i\not=l}^d x_i \be_i \,|\,x_i\in \RR \,\}.$$
Since $\bT p$ is closed in $\bU_l$, the following lemma implies that there is a one-parameter subgroup $\lambda : \CC^* \mapsto \bG$ such that
$$ \lim_{t\to 0} \lambda(t) k [v, w]\,=\, p'\in \bT p.$$
Then $\bar\lambda (t)= k^{-1} \lambda(t) k[v,w]$ defines a one-parameter subgroup and
$$\lim_{t\to 0} \bar\lambda (t) [v,w]\,=\,k^{-1} p'\in \bE.$$
This contradicts to the assumption that $[v,w]$ is numerically stable.
\end{proof}

\begin{lemm}
\label{lemm:richardson}
Let $\bT$ be an algebraic torus and $\bU$ be a $\bT$-representation. If $x\in \bU$ and if $Y$ is a non-empty $\bT$-invariant closed subset of $\overline{\bT x}\,\backslash\,
\bT x$, then there is a $y\in Y$ and a one-parameter subgroup $\lambda: \CC^* \mapsto \bT$ such that $\lambda(t) x \rightarrow y$ as $t \rightarrow 0$.
\end{lemm}
\begin{proof} This is due to Richardson. For the readers' convenience, we include a proof following \cite{paul13}.
Similar to \eqref{eq:decomp}, we have a decomposition:
$$\bU\,=\, \sum_{a\in A} \,\bU_a,~~~{\rm where}~~{\bU}_{a}\,=\,\{\,u\in \bU \, |\, t\cdot u\,=\,a(t)\, u \ , \, t \in \bT\,\}.$$
We fix a basis $\{\be_i\}_{1\le i\le d}$ of $\bU$ such that $ t \cdot \be_i = a_i (t) \be_i$ for some $a_i\in A$.\footnote{For $i\not= j$, we may still have $a_i\,=\,a_j$.}  Suppose that there are $t_\ell\in \bT$ such that $t_\ell\cdot x$ converge
to a $y \in Y$ as $\ell$ goes to $\infty$.

By rearranging the indices, we may write
$$x\,=\,\sum_{i=1}^d x_i \,\be_i,~~~y\,=\,\sum_{j=k}^d y_j \,\be_j ,$$
where $1\le k \le d$, $x_i\not=0$ and $y_j\not=0$.
Our assumption implies that $a_i (t_\ell) x_i$ converge to $0$ for $i < k$ and converge to $y_i $ for $i \ge k$.
Since $x_i\not= 0$ and $y_j\not=0$, we get
$$\lim_{\ell\to 0}  a_i (t_\ell)\,=\,0~~\forall \,i \,<\,k~~{\rm and}~~~\lim_{\ell\to 0}  a_j(t_\ell)\,=\,\frac{y_j}{x_j}\,\not=\,0~~\forall\, j\,\ge \,k.$$

Consider the quotient
$$ \pi:  \bM_\RR\,\mapsto\, W\,= \,\bM_\RR / \bM_y,~~~{\rm where}~\bM_y \,=\,\bigoplus_{j=k}^d \,\RR\cdot a_j.$$
Denote by $\Delta$ he convex hull (in $W$) of $a_i$ for $i=1,\cdots, k-1$.

We claim that $0\notin \Delta$. This can be shown as follows: If the claim is false,
then there are real constants $r_1,\cdots, r_{ k-1}\,\ge\,0$ such that
$${\rm some}~r_{i}\,>\, 0,~~~ \sum_{i=1}^{k-1} r_i\,a_i \,=\,0~~{\rm mod}~\bM_y,$$
hence, there are $c_k,\cdots,c_d$ such that
$$
\sum_{i=1}^{k-1} r_i\,a_i\,=\,\sum_{j=k}^d c_j\,a_j .
$$
Hence, for all $t\in \bT$, we have
\begin{equation}\label{product}
\prod_{i=1}^{k-1} \, |a_i(t)|^{r_i}\,=\,\prod_{j=k}^d |a_j(t)|^{c_j}.
\end{equation}
By plugging in the sequence $\{t_\ell\}$, we get a contradiction since as the left side of (\ref{product})
tends to zero while the right side does not. This proves our claim.

Using this claim and the Hyperplane Separation Theorem, one can get a linear functional $f: W \mapsto \RR$ such that
$f(\pi(a_i))\,>\,0$ for $i < k$.
Furthermore, one can choose this to be \emph{rational}.
Next we lift $f$ to $$F\,=\,f\cdot \pi: \bM_\RR \mapsto \RR.$$
Then $F$ is a \emph{rational} linear functional on $\bM_\RR$. Multiplying it by an integer, we may even assume that $F$ is integer-valued.
Therefore, it induces a one parameter subgroup $\lambda: \CC^*\mapsto \bT$ satisfying:
\begin{equation}
\label{eq:lambda}
\lim_{\ell \to 0}\,\lambda (t) \,x\,=\, \sum_{j=k}^d y_j \be_j \,=\, y.
\end{equation}
The lemma is proved.

\end{proof}

\section{A Theorem of Kempf-Ness type}

In this section, we show that the stability (resp. semistability) of pairs is equivalent to the properness (resp. lower bound) of a Kempf-Ness type functional.
As before, $\bV$ and $\bW$ are two finite dimensional complex rational representations of $\bG$ together with two non-zero $v\in \bV$ and $w\in \bW$. We equip
$\bV$ and $\bW$ with Hermitian norms.

Recall a function on $\bG$ introduced by S. Paul in \cite{paul12a}:
\begin{equation}\label{eq:KN-func}
p_{v,w}(\sigma)\,=\,\log||\sigma (w)||^2\,-\,\log||\sigma(v)||^2.
\end{equation}

The following was proved in \cite{paul12a}.

\begin{lemm}\label{lemm:distance}
\begin{equation}\label{eq:distance}
p_{v,w}(\sigma)\,=\,\log \tan^2 d(\sigma [v,w], \sigma [v,0]),
\end{equation}
where $d(\cdot,\cdot)$ is the distance function of
the Fubini-Study metric on $\mathbb{P}(\bV\oplus\bW)$.
\end{lemm}
\begin{proof} As Paul did, we will derive \eqref{eq:distance} from a formula for $d(\cdot,\cdot)$.
The Hermitian norms on $\bV$ and $\bW$ induces a Hermitian inner product $(\cdot,\cdot)$ on $\bU\,=\,\bV\oplus\bW$. We will use $||\cdot||$ to denote the
Hermitian norms on these spaces. Then for any $u,u'\in \bU$,
\begin{equation}
\label{eq:dist-1}
\cos d ([u], [u'])\,=\,  \frac{|( u,u')|}{||u|| \, ||u'||}.
\end{equation}
Now taking $[u] = \sigma [v,w]$ and $[u'] = \sigma [v,0]$, where $\sigma \in \bG$, we obtain
\begin{equation}
\label{eq:dist-2}
\cos d (\sigma [v,w], \sigma [v,0])\,=\,  \frac{||\sigma(v)||}{\sqrt{||\sigma(v)||^2 \,+ \,||\sigma(w)||^2}}.
\end{equation}
It immediately implies
$$\tan d (\sigma [v,w], \sigma [v,0])\,=\, \frac{||\sigma(w)||}{||\sigma(v)||}.$$
Then \eqref{eq:distance} follows easily.
\end{proof}

It follows from this lemma that $p_{v,w}$ is proper on $\bG$ if and only if
$$d(\sigma_i [v,w],\sigma_i [v,0])\,\rightarrow\,\frac{\pi}{2}~\,{\rm whenever}~\sigma_i [v,w] \rightarrow \overline{\bG [v,w]}\,\backslash\, \bG [v,w].$$
Similarly,
$p_{v,w}$ is bounded from below on $\bG$ if and only if there is a constant $c > 0$ such that
$$d(\sigma [v,w],\sigma [v,0])\,\ge \, c ~~{\rm on}~\bG.$$
Therefore, we have

\begin{theo}\label{th:properness}
$p_{v,w}$ is proper (resp. bounded from below) on $\bG$ if and only if $(v,w)$ is stable (resp. semistable).
\end{theo}

Similarly, we have
\begin{theo}\label{th:properness-2}
$p_{v,w}$ is proper (resp. bounded from below) along any one-parameter subgroups of $\bG$ if and only if $(v,w)$ is numerically stable (resp. numerically
semistable).
\end{theo}
Observe that for any $\sigma,\tau \in \bG$, we have
\begin{equation}\label{eq:dist-3}
d(\sigma([v,w]) , \sigma([v,0]))\,\leq\, d (\sigma([v,w]) , \tau([v,0])).
\end{equation}
So we have
\begin{coro}\label{coro:infi}
The infimum of the energy $p_{v,w}$ on $\bG$ is given by
\begin{equation}\label{eq:dist-4}
\inf_{\sigma\in \bG}\,p_{v,w}(\sigma)\,=\, \log\tan^2 d(\,\overline{\bG [v,w]}, \overline{\bG [v,0]}\,).
\end{equation}
\end{coro}

\section{K-stability condition for pairs}
The stability condition in Definition \ref{defi:pair} needs to be weakened in certain applications,
such as, in studying K\"ahler metrics of constant scalar curvature. The right condition for this purpose was the one due to
S. Paul (cf. \cite{paul13}). In this section, following \cite{paul13}, we will introduce a weaker stability condition.

For simplicity, we assume $\bG\,=\,\SL(N+1,\CC)$ and $\bV, \bW$ be two $\bG$-representations as before.
There is a natural representation $\gl(N+1,\CC)$, which consists of all $(N+1)\times(N+1)$ matrices,
by left multiplication:\footnote{The following discussions still work when $\bG$ is
replaced by a general linear algebraic group and $\gl(N+1,\CC)$ is replaced by a faithful representation of $\bG$.}
$$\bG\times \gl(N+1,\CC)\,\mapsto\, \gl(N+1,\CC): \,(\sigma, B)\,\mapsto\, \sigma B.$$
We will still denote by $\bT$ a maximal algebraic subgroup of $\bG$ and write $\gl(N+1,\CC)$ as $\gl$ for simplicity.
Let $\cN(I)\subset \bM_\RR$ be the weight polytope of the identity matrix $I$ in $\gl$. This is a standard $N$-simplex which contains the origin.
Thus we can define the degree $\deg(\bV)$ of $\bV$ by
\begin{equation}\label{eq:deg-of-V}
\deg(\bV)\,=\,\min\{\,k\in \ZZ\,|\,k > 0~{\rm and}~\cN(v) \subset k\, \cN(I)~{\rm for ~all}~0\not= v\in \bV\,\}.
\end{equation}
It follows from this definition that
$$\overline{\bG [v,I^q]}\cap \mathbb{P}(\bV\oplus \{0\})\,=\,\emptyset,$$
where $I^q\in \bU = \gl^{\otimes q}$ for $q=\deg(\bV)$, that is, $(v,I^q)$ is semistable in the sense of Definition \ref{defi:pair}.
Then, by Theorem \ref{th:properness}, we have
\begin{lemm}\label{lemm:J-1}
There is a uniform constant $c >$ such that for all $\sigma\in \bG$,
$$\deg(\bV)\,\log ||\sigma || \,\ge \,\log ||\sigma v||\,-\,c,$$
where $||\sigma I||$ denotes the Hilbert-Schmidt norm of $\sigma  \in \gl$.
\end{lemm}

\begin{defi}\label{defi:p-stable}
Let $v \in \bV\setminus\{0\}$ and $w\in \bW\setminus\{0\}$. We call $(v,w)$ K-stable if the closure of the orbit $\bG ([v,w]\times [v,I^q])$ is contained in
$$\mathbb{P}(\{0\}\oplus \bW)\times \mathbb{P}(\{0\}\oplus \bU)\,\cup\,(\mathbb{P}(\bV\oplus \bW)\backslash \mathbb{P}(\bV\oplus \{0\}))\times (\mathbb{P}(\bV\oplus \bU)\backslash
\mathbb{P}(\{0\}\oplus \bU)).$$
We call $(v,w)$ K-semistable if and only if it is semistable.
\end{defi}
\begin{rema}
\label{rema:p-stable}
The condition in the above definition actually means that the limit of a sequence $\sigma_i [v,w]$ lies in $\mathbb{P}(\{0\}\oplus \bW)$ if $\sigma_i [v, I^q]$ diverges
to $\mathbb{P}(\{0\}\oplus \bU))$. The condition described above makes it easier to see the proof of the Hilbert-Mumford-Paul criterion for P-stability
by using previous arguments.
\end{rema}

The following is the Hilbert-Mumford-Paul criterion for K-stability.
\begin{theo}\label{th:HMP-2}
Let $v, \ w$ be as above. Then $(v,w)$ is K-stable if and only if for all one parameter subgroups $\lambda$ of $\bG$,
we have $w_{\lambda}(w)\,<\, w_{\lambda}(v)$ whenever $\deg(\bV) \,w_{\lambda}(I)\,<\, w_{\lambda}(v)$.
\end{theo}
This can be proved in an identical way as we did for Theorem \ref{th:HMP}.

Similarly, we have the analogue of Theorem \ref{th:properness}
\begin{theo}\label{th:properness-3}
$(v,w)$ is P-stable if and only if $p_{v,w}$ is proper modulo $p_{v,I^q}$, i.e., for any sequence $\sigma_i\in\bG$,
$$p_{v,w}(\sigma_i) \,\rightarrow\,\infty~~{\rm whenever}~~p_{v, I^q}(\sigma_i) \,\rightarrow\,\infty,$$
where $p_{v, I^q}(\sigma_i)\,=\,\deg(\bV)\,\log ||\sigma_i || - \log ||\sigma_i v||$.
\end{theo}

\section{Resultants and hyperdiscriminants}
Let $M\subset \CC P^N$ be a complex $n$-dimensional submanifold of degree $d$
and $L$ be the restriction of ${\cal O}(1)$ to $M$. Following \cite{paul08}, we will assign
$M$ a pair of vectors $(R_M, \Delta_M)$.

Let us recall the definition of $R_M$ and $\Delta_M$. Denote by $G(k,N)$ the Grassmannian of all $k$-dimensional subspaces in
$\CC P^N$. We define
\begin{equation}
\label{eq:chow-1}
Z_M\,=\,\{\, P \in G(N-n-1,N)\,|\,P\cap M\,\not=\,\emptyset\,\}.
\end{equation}
It is known that $Z_M$ is an irreducible divisor of $G(N-n-1,N)$ with degree $d$, so there is a section in $H^0(G(N-n-1,N), {\cal O}(d))$
such that $Z_M=\{ R_M = 0\}$. Such a section pulls back to a homogeneous polynomial $R_M \in \CC[M_{(n+1)\times (N+1)}] $ of degree $(n+1) d$, where
$M_{k\times l}$ denotes the space of all $k\times l$ matrices. In fact, $R_M$ is $\SL(n+1,\CC)$-invariant which acts on $M_{n+1)\times (N+1)}$
by left multiplication. Usually, $R_M$ is called a $M$-resultant.

Next consider the Segre embedding:
$$M\times \CC P^{n-1} \subset \CC P^N\times \CC P^{n-1} \mapsto \mathbb{P}(M_{n\times (N+1)}^\vee),$$
where $M_{k\times l}^\vee$ denotes its dual space of $M_{k\times l}$. Then we define
\begin{equation}
\label{eq:chow-2}
Y_M\,=\,\{\, H\, \subset\,\mathbb{P}(M_{n\times (N+1)}^\vee)\,|\, T_p(M\times \CC P^{n-1}) \,\subset\, H~{\rm for~some}~p\,\}.
\end{equation}
One can show that $Y_M$ is a divisor in $\mathbb{P}(M_{n\times (N+1)}^\vee)$ of degree $\bar d\,=\,n d (n+1-\mu)$, where
$$\mu\,=\, \frac{c_1(M) \cdot c_1(L)^{n-1} ([M])}{c_1(L)^n([M])}.$$
Hence, there is a section in
$H^0(\mathbb{P}(M_{n\times (N+1)}^\vee), {\cal O}(\bar d))$ whose zero set is $Y_M$. Such a section corresponds to a homogeneous polynomial
$\Delta_M$ in $\CC[M_{n\times (N+1)}]$ of degree $\bar d$, referred as the hyperdiscriminant of $M$.

Now we can associate $M$ with $(R(M), \Delta(M))$ in $\bV\times \bW$ as follows:
\footnote{Since this pair was introduced by S. Paul, we may call it P-coordinate of $M$ for convenience.}
\begin{eqnarray}\label{eq:p-coor}
R(M)&=&R_M^{\bar d} \,\in \, \bV\,=\,\CC_r [M_{(n+1)\times (N+1)}],\\
\label{eq:p-coor-2}\Delta(M)&=&\Delta_M^{(n+1)d}\,\in\, \bW\,=\,\CC_r [M_{n \times (N+1)}],
\end{eqnarray}
where $r \,= \,(n+1) d \bar d$ and $\CC_r[\CC^k]$ denotes the space of homogeneous polynomials of degree $r$ on $\CC^k$.
Note that $\deg(bV)=\deg(\bW)=r$.

The automorphism group $\bG = \SL(N+1,\CC)$ of $\CC P^N$ induces actions on $\bV$ and $\bW$ in a natural way. Thus, we can have
\begin{defi}\label{defi:paul}
We call $M\subset \CC P^N$ P-stable (resp. P-semistable) with respect to the polarization $L$ if its P-coordinate
$(R(M),\Delta(M))$ is K-stable (resp. semistable) in the sense of
Definition \ref{defi:p-stable} (resp. Definition \ref{defi:pair}).
\end{defi}

\section{CM-stability as stability of pairs}

The CM stability was introduced by myself in 1996 to study the problem of K\"ahler-Einstein metrics on Fano manifolds.
It can be easily extended to any compact K\"ahler manifold $M$ polarized by an ample line bundle $L$.
In this section, following \cite{paul12a}, we reformulate
the CM-stability as a stable pair by using P-coordinates.

We assume $M\subset \CC P^N$ and $L = {\cal O}|_M$ and $\bG\,=\,\SL(N+1)$..
Given any $\sigma\in \bG$, there is a K\"ahler potential $\varphi_\sigma$ on $M$ such that
$$\sigma^*\omega_{FS}|_M\,=\,\omega_0\,+\, \sqrt{-1}\,\partial\bar\partial \varphi_\sigma,$$
where $\omega_{FS}$ denotes the Fubini-Study metric on $\CC P^N$ and $\omega_0$ is a fixed K\"ahler metric with K\"ahler class
$2\pi c_1(L)$. More precisely, we define $\varphi_\sigma$ as follows: Choose an Hermitian metric $||\cdot||_0$ on $L$ with curvature form $\omega_0$,
then we have an induced inner product on $H^0(M,L)$ by $\omega_0$ and $|\cdot ||_0$, choose an orthonormal basis $\{S_i\}_{0\le i\le N}$ with respect to this inner product, then we can set
\begin{equation}
\label{eq:K-potential}
\varphi_\sigma \,=\, \log \left (\sum_{i=0}^N \,|| \sigma (S_i) ||_0^2 \right ).
\end{equation}
Then we have a function on $\bG$:
$$\bF (\sigma) \,=\,\nu _{\omega_0} (\varphi_\sigma),$$
where $\nu_{\omega_0}$ is Mabuchi's K-energy:
$$\nu_{\omega_0}(\varphi)\,=\,-\int_0^1 \int_M \varphi \,({\rm Ric}(\omega_{t\varphi}) - \mu \,\omega_{t\varphi} )\wedge \omega_{t\varphi}^{n-1} \wedge dt,$$
where $\omega_\varphi = \omega_0\,+\, \sqrt{-1}\,\partial\bar\partial\, \varphi$. Also we define
\begin{equation}
\label{eq:func-2}
\bJ_{\omega_0}(\varphi)\,=\, \sum_{i=0}^{n-1} \frac{i+1}{n+1} \int_M
\sqrt{-1} \, \partial \varphi \wedge \overline{\partial} \varphi \wedge
\omega^i_0\wedge \omega_{\varphi}^{n-i-1}.
\end{equation}
Put $\bJ (\sigma)\,=\, \bJ_{\omega_0}(\varphi_\sigma)$.
\begin{defi}
\label{defi:CM-1}
We say $M$ CM-stable with respect to $L$ if for any sequence $\sigma_i\in \bG$,
$$\bF(\sigma_i) \rightarrow\infty~{\rm whenever}~\bJ(\sigma_i) \rightarrow \infty.$$
We say $M$ CM-semistable with respect to $L$ if $\bF$ is bounded from below.
\end{defi}

\begin{rema}
\label{rema:CM}
As we argued in \cite{tian97}, there is an algebraic formulation of the CM-stability in terms of the orbit of a lifting of $M$
in certain determinant line bundle, referred as the CM-polarization.
\end{rema}

In \cite{paul08}, S. Paul proved a remarkable formula for $\bF$ in terms of the resultant $R_M$ and the hyperdiscriminant $\Delta_M$. As a consequence,
he showed in \cite{paul12a}

\begin{theo}\label{th:cm=p}
Let $M\subset \CC P^N$, $L$ and $\bG$ be as above. Then $M$ is CM-stable (resp. CM-semistable) with respect to $L$ if and only if
$M$ is P-stable (resp. P-semistable), i.e., $(R(M),\Delta(M))$ is P-stable (resp. P-semistable).
\end{theo}
\begin{proof}
Let $\bV$ and $\bW$ be defined in \eqref{eq:p-coor} and \eqref{eq:p-coor-2} in last section.
By Theorem A in \cite{paul08}, there is a uniform constant $C$ such that for all $\sigma\in \bG$,
we have
\begin{equation}\label{eq:F-1}
|\,\bF(\sigma ) \,-\, a_n\,  p_{R(M),\Delta(M)}(\sigma) \,|\, \leq\,C,
\end{equation}
where $a_n$ is a uniform constant, $p_{R(M),\Delta(M)}$ is defined in \eqref{eq:distance} with $v= R(M)$ and $w=\Delta(M)$.

Next, we observe that the main result of \cite{paul04} gives
\begin{equation}\label{eq:J-1}
 (n+1)\,\bJ (\sigma)\,= \,  (n+1)\, \int_M \varphi_\sigma \,\omega_0^n\,-\, \log ||\sigma R_M||^2
\end{equation}
It follows
\begin{equation}\label{eq:J-2}
(n+1) \,\bar d  \,\bJ (\sigma)\,= \,  \deg (\bV) \, \int_M \varphi_\sigma \,\frac{\omega_0^n}{d}\,-\, \log ||\sigma R(M)||^2.
\end{equation}
Here we have used the fact that $\deg(\bV) = r$.

If we write $\sigma\in \SL(N+1,\CC)$ as a $(N+1)\times (N+1)$-matrix $(\vartheta_{ij})$ with determinant one, then
the Hilbert-Schmidt norm of $\sigma$ is given by
$$||\sigma||^2\,=\,\sum_{i,j=0}^{N} |\vartheta|^2.$$
Since $\{S_j\}_{0\le j\le N}$ is an orthonormal basis, we have
$$\sum_{i=0}^{N} \,\int_M ||\sum_{j=0}^{N} \vartheta_{ij} S_j ||^2 \,\omega_0^n \,=\, ||\sigma||^2 .$$
Hence, if we put $\beta_{ij} = \vartheta_{ij}/ ||\sigma||$, then we have

\vskip 0.1in
\noindent
({\bf 1}) For any $i$ between $0$ and $N$,
$$ \int_M ||\sum_{j=0}^{N} \beta_{ij} S_j ||^2 \,\omega_0^n \,\le\, 1;$$
\vskip 0.1in
\noindent
({\bf 2}) There is at least one $i'$ such that
$$ \int_M ||\sum_{j=0}^{N} \beta_{i'j} S_j ||^2 \,\omega_0^n \,\ge\, \frac{1}{N+1}.$$
By the concavity of the logarithmic function and ({\bf 1}), we have
$$\int_M\, \log \left (\sum_{i=0}^{N} ||\sum_{j=0}^N \beta_{ij} S_j ||^2\right ) \,\frac{\omega_0^n}{d} \,\le\,\log \left (
\sum_{i=0}^{N} \int_M \,||\sum_{j=0}^N \beta_{ij} S_j ||^2 \,\frac{\omega_0^n}{d} \right )\,\le\,\log (N+1).$$

On the other hand, one can deduce from ({\bf 2}) that for some $C > 0$ depending only on $M\subset \CC P^N$,
$$\int_M\, \log \left(||\sum_{j=0}^N \beta_{i'j} S_j ||^2 \right )\,\frac{\omega_0^n}{d} \,\ge\,- C.$$
In fact, by using the $\alpha$-invariant, one can show a stronger integral bound, that is, there is a uniform bound on the integral of
$||\sum_{j=0}^N \beta_{i'j} S_j ||^{-\gamma}$ for some $\gamma > 0$. 

We may assume that $C \ge \log (N+1)$. Thus, combining the above two estimates with \eqref{eq:J-2}, we get
\begin{equation}\label{eq:J-3}
|\,\bJ(\sigma) \,-\, p_{v, I^q} (\sigma)\,| \,\le\, C.
\end{equation}
The theorem follows easily from \eqref{eq:F-1}, \eqref{eq:J-3} and the definitions of CM-stability and P-stability.
\end{proof}

The arguments in the proof of Theorem \ref{th:cm=p} also yield

\begin{theo}
\label{th:ncm=np}
The P-coordinate $(R(M),\Delta(M))$ is numerically K-stable (resp. numerically semistable) if and only if $\bF$ is proper (resp. bounded from below)
along any one-parameter subgroup of $\bG$.
\end{theo}
Therefore, $M$ is CM-stable (resp. CM-semistable) with respect to $L$ if and only if
$(R(M),\Delta(M))$ is numerically K-stable (resp. numerically semistable). However, it follows from
\cite{paultian04} that $\bF$ is proper along an one-parameter subgroup $\lambda$ of $\bG$ if and only if the associated Futaki invariant
is positive \footnote{As in \cite{paul13}, we can simply define the Futaki invariant as $\bw_\lambda (\Delta(M)) - \bw_\lambda (R(M))$.
It coincides with the one introduced by Futaki and generalized by Ding-Tian and Donaldson when $\lim \lambda(t)(M)$ is reduced.}
Thus we have
\begin{coro}
\label{coro:cm=p}
$M$ is CM-stable if and only if it is K-stable.
\end{coro}
\begin{rema}
One should be able to give a direct proof of this without going through $(R(M), \Delta(M))$ by using the decomposition $\bG = \bK \bT\bK$.
The process resembles what we did in proving the Hilbert-Mumford-Paul criterion.
\end{rema}

The following corollary completes an approach suggested in \cite{tian10} and used in \cite{tian12}.
\begin{coro}\label{coro:tian12}
If $M$ is a K-stable Fano manifold and $L = K_M^{-\ell}$, then for any sequence $\{\sigma_i\}\subset \bG$ such that $M_\infty = \lim \sigma_i(M)$ is normal,
$\bF(\sigma_i)$ diverges to $+\infty$.
\end{coro}

\begin{rema}
\label{rema:tian12}
There is another functional in the study of K\"ahler-Einstein metrics on Fano manifolds:
$$F_{\omega_0}(\varphi)\,=\,\bJ_{\omega_0}(\varphi)\,-\,\int_M \varphi \,\omega_0^n \,-\,\log \left (\frac{1}{V} \int_M e^{h_0 - \varphi} \,\omega_0^n\right ),$$
where $V=\int_M \omega_0^n$ and $h_0$ is defined by
$${\rm Ric}(\omega_0) \,-\, \omega_0\,=\,\sqrt{-1}\,\partial\bar\partial \,h_0,~~~\int_M e^{h_0}\,\omega_0\,=\, V.$$
This $F_{\omega_0}$ has at most one critical point, i.e., the K\"ahler-Einstein on $M$ if it exists, and plays a similar role as the K-energy $\nu_{\omega_0}$ does.
One can prove an analogue of Corollary \ref{coro:tian12} for $F(\sigma)$ on $\bG$ by similar arguments, where $F(\sigma)=F_{\omega_0}(\varphi_\sigma)$.

In fact, in deriving the existence of K\"ahler-Einstein metrics from the K-stability, $\bF(\sigma_i)$ and $F(\sigma_i)$ are equivalent along
involved sequence $\{\sigma_i\}\subset \bG$.
\end{rema}

\section{A final remark}

Here we mention an interesting problem on stable pairs.
We consider the structure of the set of all semi-stable pairs. Let
$\bV$ and $\bW$ be two $\bG$-representations, we say $(v',w') \prec (v,w)$ if $(v',w')$ is contained in the closure
of the orbit $\bG [v,w]$ in $\mathbb{P}(\bV\oplus\bW)$. Clearly, $(v,0)\prec (v,w)$ if and only if $(v,w)$ is not semi-stable.
For any $(v,w)\in \mathbb{P}(\bV\oplus\bW)\backslash \mathbb{P}(\{0\}\oplus\bW)$, we denote by $\{v,w\}$ the set of all $(v',w')$ such that either $(v',w') \prec (v,w)$ or
$(v,w) \prec (v',w')$. It follows from standard theory on group actions that there is a unique $(\bar v,\bar w) \in \{v,w\}$ whose orbit $\bG[\bar v,\bar w]$
is closed, so it precedes anyone in $\{(v,w)\}$.  We expect

\begin{conj}
\label{conj:moduli}
Let $\cM$ be the set of all $\{(v,w)\}$ with $(v,w)$ being a semi-stable pair. Then $\cM$ is a quasi-projective
variety.
\end{conj}
Of course, this is true if $\bV$ is a trivial representation since it then becomes the situation in classical Geometric Invariant Theory.




\end{document}